\documentclass{amsart}

\usepackage{a4wide}
\usepackage{amscd}
\usepackage{amssymb}
\usepackage{amsthm}
\usepackage{amsmath}
\usepackage{times}
\usepackage{latexsym}
\usepackage{amsfonts}
\usepackage{graphicx}
\usepackage{graphicx, subfigure}
\usepackage{hyperref}
\usepackage{mathrsfs}
\usepackage{natbib}
\usepackage{enumerate}

\newtheorem{theorem}{Theorem}
\theoremstyle{plain}

\newtheorem{corollary}[theorem]{Corollary}

\newtheorem{lemma}[theorem]{Lemma}

\newtheorem{assumption}[theorem]{Assumption}

\newtheorem{proposition}[theorem]{Proposition}

\theoremstyle{remark}
\newtheorem{example}[theorem]{Example}

\numberwithin{equation}{section}
\numberwithin{theorem}{section}

\newcommand{\R}{\ensuremath{\mathbb{R}}}
\newcommand{\E}{\ensuremath{\mathbb{E}}}

\newcommand{\V}{\ensuremath{\mathcal{V}}}
\newcommand{\U}{\mathcal{U}}
\newcommand{\Hi}{\mathcal{H}}

\newcommand{\bS}{\bar{\mathcal S}}

\usepackage{color}
\definecolor{darkgreen}{rgb}{0, .5, 0}

\providecommand{\norm}[1]{\ensuremath{\lVert#1\rVert}}

\providecommand{\cadlag}{c\`adl\`ag }

\providecommand{\Levy}{L\'evy }

\newcommand{\<}{\langle}
\renewcommand{\>}{\rangle}

\allowdisplaybreaks

\title[Stochastic Volterra equations and SPDEs]{Stochastic Volterra integral equations and a class of first order stochastic partial differential equations}
\author[Benth, Detering and Kr\"uhner]{Fred Espen Benth, Nils Detering and Paul Kr\"uhner}
\address[F. E. Benth]{
Department of Mathematics \\
University of Oslo\\
P.O. Box 1053, Blindern\\
N--0316 Oslo, Norway}
\email[]{fredb\@@math.uio.no}
\urladdr{http://folk.uio.no/fredb/}
\address[N. Detering]{
Department of Statistics and Applied Probability \\
University of California, Santa Barbara\\
93106\\
USA}
\email[]{detering\@@pstat.ucsb.edu}
\urladdr{http://www.pstat.ucsb.edu/faculty/detering/}
\address[P. Kr\"uhner]{
Institute for Financial and Actuarial Mathematics \\ 
University of Liverpool \\
The Mathematical Sciences Building, Liverpool L69 7ZL\\
UK}
\email[]{P.Eisenberg\@@liverpool.ac.uk}

\date{\today}
\thanks{The authors are grateful to Bernt {{\O}}ksendal for proposing this topic. F. E. Benth acknowledges financial support from "FINEWSTOCH", funded by the Norwegian 
Research Council.}
\begin{document}
\begin{abstract}
We investigate stochastic Volterra equations and their limiting laws. The stochastic Volterra equations we consider are driven by a Hilbert space valued \Levy noise and integration kernels may have non-linear dependence on the current state of the process. Our method is based on an embedding into a Hilbert space of functions which allows to represent the solution of the Volterra equation as the boundary value of a solution to a stochastic partial differential equation. We first gather abstract results and give more detailed conditions in more specific function spaces.
%
\end{abstract}

\maketitle

\section{Introduction}
Stochastic Volterra integral equations (SVIE) appear for example in population dynamics and spread of epidemics (see \citep{GLS}) and in mathematical finance as stochastic volatility models  (see \citep{GJR}). The defining characteristic of stochastic Volterra integral equations is that they are in some way defined based on an integral of the form $\int_0^t K(t,s) d M (t)$ with a possibly stochastic integrand kernel $K$ depending on the integration horizon $t$ and some stochastic process $M$ as integrator. They have first been systematically analyzed in \citep{Berger1980,Berger19802} although specific cases appeared in the literature before (see references in \citep{Berger1980,Berger19802}).  The analysis of SVIEs has later been extended in many directions, for example to allow for a term in the equation that is not adapted \citep{Protter1990,Oksendahl1993}, for singular kernels and in relation to fractional Brownian motion \citep{Coutin2001,Decreusefond2002,Wang2008} and for equations driven by general semi-martingales \citep{Protter1985}. We also mention \citep{Jaber2017} for a treatment of Volterra processes with the state and space dependence of affine form. The recent paper \citep{AOY} deals with optimal stopping of stochastic Volterra integral equations. Stochastic Volterra integral equations in a random field setting and driven by a L\'evy basis have been considered in  \citep{Chong2017} and \citep{PHAM20183082}.

In this paper we demonstrate how existence results for a class of first order stochastic partial differential equations (SPDE) can be used to derive solutions for stochastic Volterra integral
equations of the form (the precise assumptions will be introduced below)
\begin{equation}\label{SVIE:Intro}
X(t)=x_0(t)+\int_0^t\mu(t,s,X(s))\,ds+\int_0^t\sigma(t,s,X(s-))\,dL(s)\,,
\end{equation}
with $X$ in a general separable Hilbert space $\U$ and the driving L\'evy process $L$ in another Hilbert space $\V$. Our approach rests on the observation that $x\mapsto \sigma(s+x,s,X(s-))$ can be considered as an element in a space of functions $\Hi$ mapping from $\mathbb{R}_+$, the non-negative real numbers, to $\U$. We then consider an SPDE involving the derivative operator $\partial_x$ in a way such that in the mild solution of this SPDE the shift operator (which is generated by $\partial_x$) ensures that the boundary at zero is driven by integrands of the form $\sigma(t,s,X(s-))$ and allows to retrieve the SVIE. This way the Volterra equation arises as a boundary solution to an SPDE with values in $\Hi$ and required properties for $\sigma$ are encoded in the function space $\Hi$ and properties of the shift semigroup defined on $\Hi$. In addition to showing existence of solution this method then allows us to provide results on the existence of invariant measures for the Volterra equations derived from tailor made abstract results about invariant measures of SPDEs.

The connection between Volterra dynamics and SPDEs defined on some function space is not new
in light of mild solutions of SPDEs on Hilbert space. In \citep{BE} a lifting of L\'evy semistationary processes (see 
\citep{BNBV-book}) to solutions of SPDEs 
has been utilized to develop numerical schemes
for Monte Carlo simulations of paths.  
Furthermore, as shown in a Brownian setting in \citep{ZHANG2010}, Volterra solutions can be lifted to construct mild solutions for SPDEs. 

{\bf Contribution and outline:} In Section~\ref{SPDEVolterraConnection} we analyze the above outlined embedding method systematically and derive abstract results for the conditions on the space $\Hi$ needed in order to retrieve the SVIE as a boundary solution of the SPDE. To our knowledge this is the first time that the method itself has been analyzed and treated in this generality. Understanding precisely this connection is not only of interest in itself but also relevant for applications. For example in energy markets it is natural to model the forward curve as a solution to a first order stochastic partial differential equation. The electricity spot price is then given as the boundary solution and solves a SVIE. As mentioned above, for numerical approximation of an SVIE the SPDE formulation is crucial and it is thus important to know when such formulation exists. In Section~\ref{Inv:distrib} we state conditions that ensure existence of a limiting measure for the SVIE. The question of existence of a limiting distribution is relevant in particular for applications and has not been considered for Hilbert space valued SVIE before.  In Section~\ref{examples:H} then we give an example for a possible specification of the function space $\Hi$ and derive required properties of the shift semigroup and related operators on this space. With this example at hand the solution to the Stochastic Volterra integral equation~\ref{SVIE:Intro} follows directly from recent existence results for SPDEs (see \citep{T2012} and \citep{FTT2010}) and extends previous existence results for Stochastic Volterra integral equations. Further, with this example at hand we use the abstract results from Section~\ref{Inv:distrib} to provide more specific conditions for the existence of a limiting measure both, in terms of the SPDE and the SVIE coefficients.

\section{An SPDE representation}\label{SPDEVolterraConnection}
Let $(\Omega,\mathcal{F},(\mathcal{F}_t)_{t\geq 0}, P)$ be a filtered probability space satisfying the usual conditions. Let $\U$ and $\V$ be separable Hilbert spaces. 
Let further $L$ be a square integrable \Levy process in $\V$ with $\E[L(t)]=0$ for all $t\geq 0$ and with characteristic triplet equal to $(\alpha,\mathcal{Q}_0,\nu)$ in the sense of \citep[Definition 4.28]{PZ}. Here $\mathcal{Q}_0 \in L_1^+ (\mathcal{V})$ is the covariance operator of a Wiener process, where $L_1^+ (\mathcal{V})$ is the class of non-negative trace class operators, $\nu$ is the \Levy measure of $L$ and $\alpha$ the drift. We refer the reader to \citep{PZ} for the definition of Hilbert space valued \Levy processes.

We introduce some further notations, for Hilbert spaces $\mathcal{X}$ and $\mathcal{Y}$ and $(e_k)_{k\in \mathbb{N}}$ an orthonormal basis of $\mathcal{X}$ we denote by $L (\mathcal{X},\mathcal{Y})$ the bounded linear operators from $\mathcal{X}$ to $\mathcal{Y}$ and by $L_2 (\mathcal{X},\mathcal{Y})$ the space of Hilbert Schmidt operators from $\mathcal{X}$ to $\mathcal{Y}$, i.e.
\begin{align*}
L_2 (\mathcal{X},\mathcal{Y}) := \{R \in L (\mathcal{X},\mathcal{Y}) : \norm{R}_{L_2(\mathcal{X},\mathcal{Y})}  < \infty \},
\end{align*}
where $\norm{R}_{L_2(\mathcal{X},\mathcal{Y})} :=\sum_{k\in \mathbb{N}} \norm{R e_k}_{\mathcal{Y}}^2$ is the Hilbert-Schmidt norm.
We shall further denote by $\norm{R}_{\mathrm{op}}$ the usual operator norm of $R\in L(\mathcal{X},\mathcal{Y})$ whenever the involved spaces are clear from the context. We remark that while we are mainly interested in all involved spaces to be of infinite dimensions all over results also cover the finite dimensional case.

We are concerned with adapted \cadlag solutions to the following SVIE
\begin{equation}\label{V:SDE}
X(t)=x_0(t)+\int_0^t\mu(t,s,X(s))\,ds+\int_0^t\sigma(t,s,X(s-))\,dL(s)\,,
\end{equation}
where $x_0$ is a $\mathcal U$-valued $\mathcal F_0$-measurable stochastic process, $\mu (t,s,\cdot ):\U \rightarrow \U$ and $\sigma (t,s,\cdot): \U \rightarrow L_2(\V,\U)$ are parameter functions satisfying conditions which we state below.

The solution to the above SVIE will arise as a boundary solution of a process $Y$ that lives in a larger (function) space $\Hi$. For this let  $\Hi$ be a separable Hilbert space of measurable functions $h:\mathbb R_+ \rightarrow \U $, where we use the notation $\mathbb R_+$ for the non-negative real numbers. For $t\geq 0$ denote by $\delta_t:\Hi \rightarrow \U $ the evaluation map given by $\delta_t : h \mapsto h(t)$. In order to have $\delta_t$ well defined it is crucial here that each element $h\in \Hi$ is really a function and not an equivalence class of functions which might render the definition of $\delta_t$ arbitrary. The following Assumption on $\Hi$ will be needed for our main existence result Theorem~\ref{Volterra:sol}.

\begin{assumption}\label{assump:1:space}
The function space $\Hi$ is such that 
\begin{itemize}
\item the evaluation map $\delta_0: h\mapsto h(0)$ is a bounded linear operator, 
\item the set $\{u\in\U: u =\delta_0f, f\in\Hi \text{ constant function}\}$ is closed in $\U$ and
\item the translation operator $\mathcal{S}_t:h\mapsto h(\cdot+t)$ for $t\geq 0$ is well defined and $(\mathcal{S}_t)_{t\geq 0}$ is a $C_0$-semigroup in $\Hi$ and we denote its generator by $\partial_x$ or $\partial/\partial_x$. Furthermore, $(\mathcal{S}_t)_{t\geq 0}$ is quasi-contractive, i.e.
\begin{equation*}
\norm{\mathcal S_t}_{\mathrm{op}} \leq e^{\omega t}, \forall t \geq 0
\end{equation*}
for some $\omega \in \mathbb{R}$. 
\end{itemize}
\end{assumption}
We remark that Assumption~\ref{assump:1:space} implies continuity for the evaluation maps $\delta_t = \mathcal S_t\delta_0$ for any $t\geq 0$. We assume for the rest of this section that Assumption~\ref{assump:1:space} holds and we provide an example of a specific space satisfying the assumption in Section~\ref{examples:H}.

The name $\partial_x$ for the generator is motivated by the fact that for a function $f$ in the domain of $\partial_x$ we have
 $$ \partial_xf(t) = \delta_t\lim_{r\searrow 0}\frac{\mathcal S_rf-f}{r} = \lim_{r\searrow 0}\frac{f(r+t)-f(t)}{r},\quad t\geq0 $$
 from which we see that $\partial_x$ computes the right-derivative.

We also like to remark that the closedness condition for the constant functions in $\Hi$ is satisfied under any of the following conditions:
\begin{enumerate}
  \item All constant functions are contained in $\mathcal H$,
  \item $0$ is the only constant function contained in $\mathcal H$ or
  \item $\mathcal U$ is finite dimensional.
\end{enumerate}
The reason for the closedness assumptions is to allow to embed $\mathcal U$ into $\Hi$ or, more precisely, into an enlargement of $\Hi$. This can be particularly useful as many classical spaces of functions on $\mathbb R_+$ do not include constant functions who will however be crucial in proving existence of a solution to (\ref{V:SDE}). In this case we can use the following Lemma to enrich these spaces.

\begin{lemma}\label{l:constant functions}
  There is a Hilbert space $\mathcal H^+$ which contains $\mathcal H$ as a closed subspace, which satisfies Assumption \ref{assump:1:space} and such that there is a continuous linear map $\pi:\mathcal U\rightarrow\mathcal H^+$ with $\pi u(t)=u$ for any $u\in \mathcal U$, $t\geq 0$.
\end{lemma}
\begin{proof}
 Let $\mathcal P$ be the set of constant functions from $\mathbb R_+$ to $\mathcal U$ and $\|f\|_\mathcal P:=\|f(0)\|_{\mathcal U}$ be the push-forward norm $\pi:\U\rightarrow \mathcal P,u\mapsto (t\mapsto u)$. Note that $\pi$ as a mapping to $\mathcal P$ is a bijective isometry from $(\U,\|\cdot\|_\U)$ to $(\mathcal P\|\cdot\|_{\mathcal P})$ by construction.
 Let $\mathcal C:=\mathcal P\cap \mathcal H = \pi(\delta_0(\mathcal P\cap \mathcal H))$ which is closed in $\mathcal P$ because $\delta_0(\mathcal P\cap \mathcal H)$ is closed in $\U$ by assumption. Also, note that $\mathcal C$ is closed in $(\Hi,\|\cdot\|_\Hi)$ because it is the set of constant functions in $\Hi$ and the point evaluations $(\delta_t)_{t\geq 0}$ are continuous and separating. Let $\mathcal B$ be the orthogonal complement of $\mathcal C$ in $(\mathcal P\|\cdot\|_{\mathcal P})$ and define
  $$ \mathcal H^+ := \Hi \oplus \mathcal B. $$
 We define the norm
  $$ \|h+b\|^2_{\Hi^+} := \|h\|^2_{\Hi} + \|b\|^2_{\mathcal P},\quad h+b \in \mathcal H\oplus\mathcal B. $$
  Then $(\Hi^+,\|\cdot \|_{\Hi^+})$ is a Hilbert space and $\mathcal H$, $\mathcal B$ are orthogonal complements by construction. For $h+b \in \mathcal H\oplus\mathcal B$ we have 
   $$ \|\delta_0(h+b)\|^2_{\mathcal U} \leq 2\|h(0)\|_{\mathcal U}^2+2\|b(0)\|^2_{\mathcal U} \leq 2(\|\delta_0\|^2+1)\|h+b\|^2_{\Hi^+} $$
   where we used orthogonality for the last inequality. Thus, $\delta_0$ is a bounded linear operator and its range is $\mathcal U$ which is closed. Since $\delta_0|_{\mathcal P}$ is bounded relative to the $\|\cdot\|_{\Hi^+}$-norm we find that its inverse $\pi$ has a closed graph. The closed graph theorem yields continuity of $\pi$.
   
   Now it remains to see that $\Hi^+$ satisfies Assumption \ref{assump:1:space}. We already proved continuity of $\delta_0$. The set $\{u\in \U:u=\delta_0 f, f\in\Hi^+\text{ constant function}\} = \U$ by construction of $\Hi^+$. We now inspect the behaviour of the shift semigroup $(\mathcal S_t)_{t\geq 0}$. Since the functions $b\in\mathcal B$ are constant we find that $\mathcal S_tb=b$ for all $t\geq 0$. Also, for $h+b\in\Hi\oplus\mathcal B$ we have $\mathcal S_th \in\mathcal H$ and, hence, it is orthogonal to $b = \mathcal S_tb$. For this reason we find
    $$ \|\mathcal S_t(h+b)\|^2_{\Hi^+} = \|\mathcal S_th+b\|^2_{\Hi^+} = \|\mathcal S_th\|^2_{\Hi^+}+\|b\|^2_{\Hi^+} \leq \max\{1,\|\mathcal S_t\|_{\mathrm{op}}^2\}\|h+b\|^2_{\Hi^+}.$$
    Thus, $(\mathcal S_t)_{t\geq 0}$ is a quasi-contractive semigroup and we have
     $$ \mathcal S_t(h+b) = \mathcal S_th + b \rightarrow h+b,\quad t\searrow 0 $$
    and, hence, $(\mathcal S_t)_{t\geq 0}$ is a $C_0$-semigroup on $\Hi^+$.
\end{proof}

In order to make sense out of the Volterra Equation \eqref{V:SDE} we need some more assumptions.

\begin{assumption}\label{assump:2:functions}
The coefficient functions $\mu,\sigma$ are such that
\begin{itemize}
\item for each fixed $(t,u)\in\R_+\times\U$, the functions $x\mapsto\mu(t+x,t,u)$ and $x\mapsto\sigma(t+x,t,u)$ are elements of $\Hi$ and $L(\V,\Hi)$ respectively,
\item the mappings $\mathbb R_+\times\mathcal U\ni (t,u)\mapsto \mu(t+\cdot,t,u)\in\mathcal H$ and 
$\mathbb R_+\times\mathcal U\ni(t,u)\mapsto\sigma(t+\cdot,t,u)\in L(\mathcal V,\mathcal H)$ are measurable.
\end{itemize}
\end{assumption}
We now define the functions $a:\R_+\times\Hi \rightarrow\Hi$ and $b:\R_+\times\Hi \rightarrow L(\V, \Hi)$ by
\begin{align}
a(t,h)&=\mu(t+\cdot,t,\delta_0 h), \label{def-a}\\
b(t,h)&=\sigma(t+\cdot,t,\delta_0 h) \label{def-b}\,.
\end{align}
Continuity of $\delta_0$ and Assumption~\ref{assump:2:functions} yield that $a:\mathbb R_+\times\mathcal H\rightarrow\mathcal H$ and $b:\mathbb R_+\times\mathcal H\rightarrow
L(\mathcal V,\mathcal H)$ are measurable functions.

Related to the SVIE is the following class of first order SPDEs
\begin{equation}
\label{eq:hyp-spde}
dY(t)=\left(\partial_xY(t)+a(t,Y(t))\right)\,dt+b(t,Y(t))\,dL(t)\,,
\end{equation}
with $Y(0)$ being given as an $\mathcal F_0$-measurable $\Hi$-valued random variable.

We shall need some standard Lipschitz and growth conditions.
\begin{assumption}\label{Ass:Lip}
We say that functions $a:\R_+\times\Hi \rightarrow\Hi$ and $b:\R_+\times\Hi \rightarrow L(\V, \Hi)$ fulfil a {\em Lipschitz and linear growth condition}, if there exist measurable functions $L_{a}, L_{b},K_a, K_b: \mathbb{R}_
+ \rightarrow \mathbb{R}_+$ which are bounded on compacts and such that 
\begin{align}
\norm{a(t,h_1) - a(t,h_2)}_{\Hi} &\leq   L_a(t) \norm{h_1-h_2}_{\Hi}\label{lip:a}\\
\norm{b(t,h_1) - b(t,h_2)}_{\mathrm{op}} &\leq  L_b(t) \norm{h_1-h_2}_{\Hi}\label{lip:b}
\end{align}
for all $t\in\mathbb R_+$ and all $h_1,h_2 \in \Hi$ and moreover 
\begin{align}
\norm{a(t,{\bf 0_{\Hi}})}_{\Hi} &\leq   K_a (t)\label{grows:a}\\
\norm{b(t,{\bf 0_{\Hi}})}_{\mathrm{op}} &\leq  K_b (t) \label{grows:b}
\end{align}
for all $t\in \mathbb{R_+}$ and $h \in \Hi$. By ${\bf 0_{\Hi}}$ we mean the zero element of $\Hi$, i.e.\ the function which is constant zero.
\end{assumption}
Note that from the above assumption it follows directly by the triangular inequality that $\norm{a(t,h)}_{\Hi} \leq L_a(t)  \norm{h}_{\Hi}  + K_a (t) $ and $\norm{b(t,h)}_{\mathrm{op}} \leq  L_b(t)  \norm{h}_{\Hi}  + K_b (t) $, which explains the name {\em linear growth condition}. Of course, when looking for solutions for (\ref{V:SDE}) it is more natural to state assumptions on the functions $\mu$ and $\sigma$ directly and we will do so for a particular choice of $\Hi$ in Section~\ref{param:conditions}.

By \citep[Theorem 8.8]{FTT2010} 
under the linear growth and Lipschitz condition, for every $\Hi$-valued $\mathcal F_0$-measurable square-integrable random variable $x_0$ there exists a unique mild solution of \eqref{eq:hyp-spde} given by the integral equation 
\begin{equation}\label{SPDE-sol}
Y(t)=\mathcal{S}_t Y(0)+\int_0^t\mathcal{S}_{t-s}a(s,Y(s))\,ds+\int_0^t\mathcal{S}_{t-s}b(s,Y(s-))\,dL(s)
\end{equation}
with $Y(0)=x_0$ and this solution $Y$ is \cadlag and adapted. The first main result of the paper follows now and shows that the boundary of this solution solves the SVIE in Eq.~\eqref{V:SDE}.

\begin{theorem}\label{Volterra:sol}
Suppose that the coefficient functions satisfy Assumption \ref{assump:2:functions} and that $x_0$ is an $\Hi$-valued $\mathcal F_0$-measurable square-integrable random variable. Let $a$ and $b$ be as defined in \eqref{def-a} and \eqref{def-b}, and assume that the functions fulfil Assumption~\ref{Ass:Lip}. Then there is a unique adapted \cadlag solution $X$ to the SVIE \eqref{V:SDE}. Moreover, this solution satisfies
 $$ \E [\sup_{t\in [0,T]} \norm{X(t)}^2_{\mathcal{U}} ] < \infty $$
 for any $T>0$ and it is given by $X(t) = \delta_0Y(t)$ where $Y$ is the solution to the SPDE \eqref{eq:hyp-spde} with $Y(0)=x_0$. 
\end{theorem}
\begin{proof}
 {\em Step 1; construction of a solution:}
  First we observe that as $Y(0)=x_0$, $Y(0)$ is an $\mathcal F_0$-measurable square integrable random variable with values in $\Hi$ from the assumption on $x_0$. We define $X(t):=\delta_0Y(t)$ and apply $\delta_0$ to the representation \eqref{SPDE-sol}. Note that continuous linear operators can always be pushed into Bochner integrals and the stochastic integral \citep[Proposition 3.15(ii), Theorem 8.7(v)]{PZ}. Hence, we find that
  \begin{align*}
      X(t) &= \delta_0(\mathcal S_tY(0)) + \int_0^t \delta_0\mathcal S_{t-s}a(s,Y(s))ds+\int_0^t \delta_0\mathcal S_{t-s}b(s,Y(s-))dL(s)\\
           &= x_0(t) + \int_0^t \mu(t,s,X(s))ds + \int_0^t \sigma(t,s,X(s-)) dL(s) 
  \end{align*}
  for any $t\geq 0$. Thus, $X$ is a solution to the SVIE \eqref{V:SDE}. By \citep[Theorem 8.8]{FTT2010} the solution $Y$ to the SPDE \eqref{eq:hyp-spde} satisfies
  \begin{equation}\label{square:int:sol}
\E [ \sup_{t \in [0,T]} \norm{Y(t)}_{\Hi}^2]<\infty
\end{equation}
for any $T>0$. Hence, we find
 $$ \E [ \sup_{t \in [0,T]} \norm{X(t)}_{\mathcal U}^2] \leq \|\delta_0\|^2_{\mathrm{op}} \E [ \sup_{t \in [0,T]} \norm{Y(t)}_{\Hi}^2]<\infty $$
 for any $T>0$ which proves that there is a solution with the required integrability condition. According to \citep[Theorem 4.5. (1)]{T2012} $Y$ has \cadlag paths and, hence, $X$ has \cadlag paths.
 
 {\em Step 2; uniqueness of solutions:}
For the remainder of the proof, let $X$ be any adapted \cadlag solution to the SVIE \eqref{V:SDE}. We define the stopping times $\tau_N:=\inf\{t\geq 0: \|X(t)\|_\U\geq N\}$ for any $N\geq 0$ and note that $\tau_N\rightarrow\infty$ for $N\rightarrow\infty$ due to the path property of $X$. By Lemma \ref{l:constant functions} we may assume that the embedding $\pi:\U\rightarrow \Hi,u\mapsto (t\mapsto u)$ is an everywhere defined continuous linear operator (where we might possibly have to replace $\Hi$ by a larger space). \citep[Theorem 1.8.1]{SN} yields that there is a Hilbert space $\bar\Hi$ which contains $\Hi$ as a closed subspace, its norm restricted to $\Hi$ is the norm of $\Hi$, and such that there is a $C_0$-group $\bS$ defined on $\bar\Hi$ such that
 $$ \mathcal S_t = \Gamma_\Hi(\bS_t)|_{\Hi},\quad t\geq 0$$
where $\Gamma_\Hi:\bar\Hi\rightarrow\Hi$ is the orthogonal projection. We also use the notations $\bar\delta_0 := \delta_0\Gamma_\Hi$, $\bar a(s,h) := \mu(s+\cdot,s,\bar\delta_0 h)$ and $\bar b(s,h) := \sigma(s+\cdot,s,\bar\delta_0 h)$ for $s\geq 0$, $h\in\bar\Hi$. Note that $\bar a$, $\bar b$ have values in $\Hi$. Define 
$$ Z_N(t) := \bS_{t\wedge\tau_N}x_0 + \bS_{t\wedge\tau_N}\int_0^{t\wedge\tau_N} \bS_{-s}\bar a(s,\pi X(s))ds + \bS_{t\wedge\tau_N}\int_0^{t\wedge\tau_N} \bS_{-s}\bar b(s,\pi X(s-))dL(s),\quad t\geq 0$$
  where the integrals exist because the integrands are bounded. Fix $t\geq 0$ and define $A_N:=\{t < \tau_N\}$. We find
  \begin{align*}
     Z_N(t) 1_{A_N} &= 1_{A_N}\bS_{t\wedge\tau_N}x_0 + 1_{A_N}\bS_{t\wedge\tau_N}\int_0^{t\wedge\tau_N} \bS_{-s}\bar a(s,\pi X(s))ds \\
     &\quad+ 1_{A_N}\bS_{t\wedge\tau_N}\int_0^{t\wedge\tau_N} \bS_{-s}\bar b(s,\pi X(s-))dL(s) \\
       &= 1_{A_N}\bS_{t}x_0 + 1_{A_N}\bS_{t}\int_0^{t} \bS_{-s}\bar a(s,\pi X(s))ds + 1_{A_N}\bS_{t}\int_0^{t} \bS_{-s}\bar b(s,\pi X(s-))dL(s) \\
       &= \left(\bS_{t}x_0 + \int_0^{t} \bS_{t-s}\bar a(s,\pi X(s))ds + \int_0^{t} \bS_{t-s}\bar b(s,\pi X(s-))dL(s)\right)1_{A_N}
  \end{align*}
  and, hence, 
   $$ \bar\delta_0Z_N(t)1_{A_N} = X(t)1_{A_N}. $$
  Since the value of $a,b$ depend only on the initial value of the inserted function we find that
 $$ Z_N(t) := \bS_{t\wedge\tau_N}x_0 + \bS_{t\wedge\tau_N}\int_0^{t\wedge\tau_N} \bS_{-s}a(s,Z_N(s))ds + \bS_{t\wedge\tau_N}\int_0^{t\wedge\tau_N} \bS_{-s}b(s,Z_N(s-))dL(s),\quad t\geq 0. $$
 Thus, $Z_N$ is the $\tau_N$-stopped solution of the SPDE \eqref{eq:hyp-spde}, i.e.\ $Z_N(t)= Y(t\wedge\tau_N)$ for any $t\geq 0$ where $Y$ is the unique $\bar\Hi$-valued solution of the SPDE \eqref{eq:hyp-spde}. We find that
  $$ \E[\sup_{0\leq s\leq t} \|Z_N(s)\|_{\bar\Hi}^2] \leq \E[\sup_{0\leq s\leq t} \|Y(s)\|_{\bar\Hi}^2] < \infty. $$
  The sequence $(A_N)_{N\in\mathbb N}$ is an increasing and exhausting sequence of sets. The monotone convergence theorem yields 
   \begin{align*}
      \E[\sup_{0\leq s\leq t} \|X(s)\|_{\bar\Hi}^2] &\leftarrow \E[\sup_{0\leq s\leq t} \|X(s)\|_{\U}^21_{A_N}]  \\
       &\leq \| \bar\delta_0 \|^2_{\mathrm{op}}\E[\sup_{0\leq s\leq t}  \|Z_N(s)\|_{\bar\Hi}^21_{A_N}] \\
       &\leq \| \bar\delta_0 \|^2_{\mathrm{op}}\E[\sup_{0\leq s\leq t}   \|Y(s)\|_{\bar\Hi}^21_{A_N}] \\
       &\leq \| \bar\delta_0 \|^2_{\mathrm{op}}\E[\sup_{0\leq s\leq t}  \|Y(s)\|_{\bar\Hi}^2] < \infty.
      \end{align*}
     Thus, we find $\E[\sup_{0\leq s\leq t} \|X(s)\|_{\bar\Hi}^2] < \infty$ and 
      $$X(t) = \lim_{N\rightarrow\infty} X(t) 1_{A_N} = \lim_{N\rightarrow\infty} \bar\delta_0Z_N(t) 1_{A_N} = \lim_{N\rightarrow\infty} \delta_0Y(t\wedge\tau_N) 1_{A_N} = \delta_0Y(t). $$
     Since $X$ and $\delta_0Y$ have \cadlag paths we find that $X=\delta_0Y$.
\end{proof}

\section{Invariant distributions}\label{Inv:distrib}
In this section we investigate existence of an invariant distribution for homogeneous SVIEs, i.e.\ we consider equations of the following type
\begin{equation}\label{V:SDE:hom}
X(t)=x_0+\int_0^t\mu(t-s,X(s))\,ds+\int_0^t\sigma(t-s,X(s-))\,dL(s)\,,
\end{equation}
where $x_0$ is a square-integrable, $\mathcal F_0$-measurable and $\U$-valued random variable, $\mu:\mathbb R_+\times \U\rightarrow \U$ and $\sigma:\mathbb R_+\times U\rightarrow L(\mathcal V,\U)$. This implies that $a$ and $b$ do not depend on time, namely $a(h)=\mu(\cdot,\delta_0h)$ and $b(h) = \sigma(\cdot,\delta_0h)$ for any $h\in \Hi$. We will assume for the remainder of this section that our Assumptions \eqref{assump:2:functions}, \eqref{Ass:Lip} are satisfied and that $L$ is a square integrable L\'evy process with $\E[|L(1)|_\V^2] = 1$ and zero mean.\footnote{If $\mathcal R$ is the RKHS of $L$, cf.\ \citep[Definition 7.2]{PZ}, then $\mathcal R\subseteq \mathcal V$ and for $T\in L(\V,\U)$ one has $\|T\|_{L_2(\mathcal R,\U)} \leq \|T\|_{\mathrm{op}}$.}

We shall focus on abstract results for a generic function space $\Hi$ here. In Section~\ref{examples:H} our Theorems \ref{inv:distribution} and \ref{inv:distribution 2} provide more specific existence results of limiting laws. 

One way to guarantee existence of a limiting measure for the SPDE is to ensure that $\|\mathcal S_t\|_{\mathrm{op}} \leq e^{- \beta t}$ for $\beta>0$ large enough (its required magnitude depending on the Lipschitz coefficients of the SPDE). However, in order to derive the SVIE we need to ensure that $\Hi$ contains constant functions. For constant $h \in \Hi$ the shift operator acts as the identity, i.e. $\mathcal S_t h=h$ and thus $\|\mathcal S_t\|_{\mathrm{op}} \geq 1$ and the standard conditions are not fullfilled. We provide two conditions under which one can still ensure existence of a limiting measure for the SPDE and as a result for the SVIE. The first one captures the case where the coefficients are orthogonal to the subspace generated by constant functions. In this case the influence of the starting value persists in the limiting measure as the coefficient functions leave it untouched. The second results covers a setting where $h \in \Hi$ can be split into an orthogonal sum of two subspaces, one on which $\mathcal S_t$ has nice contraction properties and one on which it is only quasi-contractive but the drift coefficient mean reverts towards $0$ on that subspace. In this case the limiting measure is independent of the starting value.
\begin{proposition}\label{prop:limit}
  Let $\mathcal C \subseteq \{h\in \Hi: \partial_xh=0\}$ such that its orthogonal complement $\mathcal B$ is invariant under the shift semigroup 
  $(\mathcal S_t)_{t\geq 0}$. Further we assume that there is $\alpha>0$, $L_a,L_b>0$ such that
  \begin{enumerate}
     \item[(i)] $a$ is $\mathcal B$-valued and $b$ is $L(\V,\mathcal B)$-valued,
    \item[(ii)] $\|\mathcal S_t|_{\mathcal B}\|_{\mathrm{op}} \leq e^{-\alpha t/2}$,
\item[(iii)] $\norm{a(h_1) - a(h_2)}_{\Hi} \leq   L_a \norm{h_1-h_2}_{\Hi}$, 
\item[(iv)] $\norm{b(h_1) - b(h_2)}_{\mathrm{op}} \leq  L_b \norm{h_1-h_2}_{\Hi}$ and
\item[(v)] $2L_a+L_b^2 < \alpha$
  \end{enumerate}
  for any $h_1,h_2\in\mathcal B$. Then for any $x_0\in \U$ there is a limiting distribution $\nu_{x_0}$ for the solution to the SVIE
  $$ X(t) = x_0 + \int_0^t \mu(t-s,X(s)) ds + \int_0^t \sigma(t-s,X(s)) dL(s), $$
  i.e.\ $X(t) \rightarrow \nu_{x_0}$ in law for $t\rightarrow \infty$.
  
  If $\mathcal C=\{0\}$, then the limiting distribution $\nu$ does not depend on the distribution of $x_0$ and it is an invariant law for $X$.
\end{proposition}
\begin{proof}
  Let $\pi_\mathcal B:\Hi\rightarrow\mathcal B$ be the orthogonal projection and $z_0 := \pi_\mathcal Bx_0$ where we identify $x_0$ with the constant function $t\mapsto x_0$. Note that $\pi_\mathcal B\mathcal S_t = \mathcal S_t\pi_{\mathcal B}$ for any $t\geq 0$ because $\mathcal S_t h =h$ for any $h\in\mathcal C$ and $\mathcal B$ is $\mathcal S$-invariant. Let $Y$ be the mild solution to the SPDE \ref{eq:hyp-spde}, i.e.
   $$ Y(t) = x_0 + \int_0^t \mathcal S_{t-s}a(Y(s)) ds + \int_0^t \mathcal S_{t-s}b(Y(s)) dL(s)$$
  and by condition (i)
   $$ Z(t) := \pi_{\mathcal B}Y(t) = z_0 + \int_0^t \mathcal S_{t-s}a(Y(s)) ds + \int_0^t \mathcal S_{t-s}b(Y(s)) dL(s).$$
  We see that $Y(t) = Z(t) + x_0-z_0$ for any $t\geq 0$ which yields
   $$ Z(t) = z_0 + \int_0^t \mathcal S_{t-s}a(Z(s) + z_0-x_0) ds + \int_0^t \mathcal S_{t-s}b(Z(s)+z_0-x_0) dL(s).$$
  We now like to verify \citep[Theorem 16.5]{PZ} for $Z$ on $\mathcal B$. First note that \citep[Theorem 16.5]{PZ} does not allow for stochastic coefficients. However, our stochastic dependency is on $\mathcal F_0$ only and the increments of the driving L\'evy process $L$ are $\mathcal F_0$-independent. A simple conditioning argument allows to use $\mathcal F_0$-dependent coefficients in \citep[Theorem 16.5]{PZ}. 
  
  Now, let $A_n$ be the $n$-th Yosida approximation of $\mathcal S$ on $\mathcal H$, i.e.\  
   $$ A_n h := n^2\int_0^{\infty}e^{-nt}(\mathcal S_t h-h) dt,\quad h\in\mathcal B$$
  and condition (ii) yields that when restricting to $\mathcal B$ we have
  \begin{eqnarray}
   \<A_nh,h\> &=& n^2\int_0^{\infty}  e^{-nt} \< (\mathcal S_t h-h) , h\>dt \nonumber  \\
   &  \leq & n^2\int_0^{\infty} e^{-nt} ( \| \mathcal S_t  |_{\mathcal B}\|_{\mathrm{op}} -1 )\| h \|_\Hi^2 dt \nonumber \\ 
   &  \leq & n^2\int_0^{\infty} e^{-nt} ( e^{-\alpha t/2} -1 ) dt \| h \|_\Hi^2 \nonumber \\
   &  \leq & -\frac{\alpha}{2+\alpha/n}\|h\|^2_\Hi, \quad h\in\mathcal B. \nonumber
  \end{eqnarray}
 Due to conditions (iii), (iv) and (v), we find with $\epsilon := \alpha-(2L_a+L_b^2)>0$ and $n\in\mathbb N$ larger than $\max\{2\alpha,2\alpha^2/\epsilon\}$ that
 \begin{align*} 
  &2\< A_n(g-h) + a(g)-a(h), g-h\> + \|b(g) - b(h)\|^2_{L(\V,\Hi)} \\
  &\qquad\qquad\leq -2\frac{\alpha}{2+\alpha/n}\|g-h\|_{\Hi}^2 + (2L_a+L^2_b)\|g-h\|^2_{\Hi} \\
  &\qquad\qquad\leq \frac{-\epsilon + (\alpha^2/n-\epsilon)}{2+\alpha/n}\|g-h\|_{\Hi}^2 \\
  &\qquad\qquad\leq \frac{-\frac{3}{2}\epsilon}{2+\alpha/n}\|g-h\|_{\Hi}^2 \\
  &\qquad\qquad\leq -\epsilon /2 \|g-h\|_{\Hi}^2
  \end{align*}
for any $g,h\in\mathcal B$. Thus, the requirements of \citep[Theorem 16.5]{PZ} are met and, hence, there is a limiting law $\mu$ for $Z(t)$ when $t\rightarrow\infty$ which does not depend on the initial law of $Z$. Since $X(t) = \delta_0(x_0-z_0+Z(t)) = x_0 + \delta_0(Z(t)-z_0)$ we find that $X$ has a limiting law, depending on $x_0$.

For the last part of the statement we may now assume additionally that $\mathcal C=\{0\}$. Then $\mathcal B=\Hi$ and $\pi_\mathcal B$ is the identity. Thus, $x_0-z_0=0$ which yields
 $$ X(t) = \delta_0Z(t) \rightarrow \nu := \mu^{\delta_0},\quad \text{ in law when }t\rightarrow 0. $$
 
 Now, let the law of $Z(0)$ be $\mu$ and $X := \delta_0Z$. Then $Z(t)$ has the same law as $Z(0)$ for any $t\geq 0$, the law of $X(0)$ is $\nu$, $X(0)$ is the unique solution to the SVIE \eqref{V:SDE:hom} and the law of $X(t)$ is the pushforward law of $Z(t)$ under $\delta_0$ and, hence, this law is $\nu$. Consequently, $\nu$ is an invariant law for the SVIE \eqref{V:SDE:hom}.
\end{proof}

\begin{proposition}\label{prop:perm abstract}
  Let $\pi_0$, $\pi_1$ be orthogonal projections on $\Hi$ with $\pi_0+\pi_1$ equal to the identity operator. We assume that there is $L_a,L_b,\beta,\gamma\geq 0$ such that
   \begin{enumerate}[(i)]
    \item $\|\pi_0 \mathcal S_t g\|_{\Hi} \leq e^{\gamma t/2}\|\pi_0 g\|_\Hi$,
    \item $\|\pi_1\mathcal S_t g\|_{\Hi} \leq e^{-\beta t}\|\pi_1 g\|_\Hi$,
    \item $\<\pi_0( a(g) - a(h)),g-h\> \leq -\beta \|\pi_0(g-h)\|^2_\Hi$,
\item $\norm{\pi_1(a(g) - a(h))}_{\Hi} \leq   L_a \norm{g-h}_{\Hi}$, 
\item $\norm{b(g) - b(h)}_{\mathrm{op}} \leq  L_b \norm{g-h}_{\Hi}$ and
\item $\gamma+2L_a+L_b^2 < 2\beta$
   \end{enumerate}
   for any $t\geq 0$ and $g, h\in\Hi$.
   
 Then there is a limiting distribution $\nu$ for the solution to the SVIE
  $$ X(t) = x_0 + \int_0^t \mu(t-s,X(s)) ds + \int_0^t \sigma(t-s,X(s)) dL(s), $$
  i.e.\ $X(t) \rightarrow \nu$ in law for $t\rightarrow \infty$ and $\nu$ does not depend on the initial value.
\end{proposition}
\begin{proof}
  Let $Y$ be the mild solution to the SPDE \ref{eq:hyp-spde}, i.e.
   $$ Y(t) = x_0 + \int_0^t \mathcal S_{t-s}a(Y(s)) ds + \int_0^t \mathcal S_{t-s}b(Y(s)) dL(s)$$
  for $t\geq 0$. We now like to verify the conditions of \citep[Theorem 16.5]{PZ} for $Y$. To this end, let $A_n$ be the $n$-th Yosida approximation of $\mathcal S$. Condition (i) yields together with the Cauchy-Schwarz inequality and $\< \pi_0 h , g\> = \< \pi_0 h , \pi_0 g\>$ for every $h,g \in \Hi$ that 
  \begin{eqnarray}
  \<\pi_0 A_nh,h\> &=& n^2\int_0^{\infty}  e^{-nt} \< \pi_0 (\mathcal S_t h-h) , h\>dt \nonumber  \\
  &=& n^2\int_0^{\infty}  e^{-nt} \< \pi_0 (\mathcal S_t h-h) , \pi_0h\>dt \nonumber  \\
   &  \leq & n^2\int_0^{\infty} e^{-nt} ( \| \pi_0  \mathcal S_t h \|_\Hi \| \pi_0h \|_\Hi  -\| \pi_0h \|_\Hi^2 dt \nonumber \\ 
   &  \leq & n^2\int_0^{\infty} e^{-nt} ( e^{\gamma t/2} -1 ) dt \| \pi_0 h \|_\Hi^2 \nonumber \\
  &\leq & \frac{\gamma}{2-\gamma/n}\| \pi_0 h\|^2_\Hi, \quad h\in\Hi \label{yosida:estimate}
 \end{eqnarray}
  whenever $n$ is such that $\gamma/n<2$. Similarly we obtain from (ii) that 
  \begin{equation}\label{yosida:estimate2}
   \<\pi_1 A_nh,h\> \leq   \frac{-\beta}{1-\beta/n}\| \pi_1 h\|^2_\Hi, \quad h\in\Hi
  \end{equation}
  for $\beta/n <1$.
  Moreover, from (iv) it follows by Cauchy-Schwarz 
  \begin{equation}\label{pi1:scalar}
  \<\pi_1a(g)-\pi_1 a(h),g-h\> \leq \| \pi_1a(g)-\pi_1 a(h)  \|_\Hi  \|g-h \|_\Hi \leq L_a \norm{g-h}_{\Hi}^2
  \end{equation}
  
  Define now $\epsilon := (2\beta-\gamma-2L_a-L_b^2)/2$ (which is strictly positive by (vi). 
Using (iii), (v), (\ref{yosida:estimate}), (\ref{yosida:estimate2}), (\ref{pi1:scalar}) and the fact that $\|\pi_1 (g-h)\|_\Hi \leq \|g-h\|_\Hi$ we obtain 
   \begin{align*}
      2\< &A_n(g-h)+a(g)-a(h),g-h\> + \|b(g)-b(h)\|_\Hi^2 \\
      & = 2\<  \pi_0 A_n(g-h),g-h\> + 2\< \pi_1 A_n(g-h),g-h\> + 2 \<\pi_0a(g)-\pi_0a(h),g-h\> \\ 
      &+ 2\<\pi_1a(g)-\pi_1 a(h),g-h\>  + \|b(g)-b(h)\|_\Hi^2 \\
      &\leq \frac{2\gamma}{2-\gamma/n}\| \pi_0 (g-h)\|^2_\Hi -  \frac{2\beta}{1-\beta/n}\| \pi_1 (g-h) \|^2_\Hi - 2\beta \|\pi_0(g-h)\|^2_\Hi+2 L_a\|g-h\|^2_\Hi + L_b^2\|g-h\|^2_\Hi  \\
      &\leq \left(\frac{\gamma}{1-\gamma/(2n)}-2\beta+2L_a+L_b^2 \right) \|g-h\|^2_\Hi -  \frac{2\beta^2}{n-\beta}\| \pi_1 (g-h) \|^2_\Hi   \\
      &\leq  \left(\frac{\gamma}{1-\gamma/(2n)}-2\beta+2L_a+L_b^2 \right) \|g-h\|^2_\Hi  \leq -\epsilon\|g-h\|^2_\Hi  
   \end{align*}
   for $n$ large. Thus, \citep[Theorem 16.5]{PZ} yields that $Y$ has a limiting law which does not depend on the initial value and, hence, $X(t) = \beta_0Y(t)$ has a limiting law which does not depend on its initial value.
\end{proof}
In the next section we provide more specific conditions for our example space $\Hi$ and further analyze the interplay of Lipschitz conditions and contractivity properties of the semigroup in relation to the constant functions.

\section{An example of a function space $\Hi$}\label{examples:H}

We shall now provide a specification of $\Hi$ that allows us to consider Volterra SDEs in general separable Hilbert spaces $\U$. Our example is the extension in \citep{BE2015} of the Filipovi\'c space introduced by \citep{Fili-lectures}. We denote by $L^1_{loc} (\mathbb{R}_+, \U)$ the space of locally Bochner-integrable functions from $\mathbb{R}_+$ to $\U$ and by $AC(\mathbb{R}_+,\U)$ the space of absolutely-continuous functions from $\mathbb R_+$ to $\U$, i.e.\ $f\in AC(\mathbb{R}_+,\U)$ if and only if there is a function $g\in L^1_{loc} (\mathbb{R}_+, \U)$ with $f(x)-f(y) = \int_y^xg(s)ds$ for any $0\leq y \leq x$. If $f\in AC(\mathbb{R}_+,\U)$ is given, then the function $g\in L^1_{loc} (\mathbb{R}_+, \U)$ is $ds$-a.e.\ unique and we write $f':=g$ for a version. Whenever $f'$ has a continuous version we mean by $f'$ the unique continuous version. Following \citep{BE2015} we define the space $\Hi_w$ of $\U$-valued {\em smooth} functions. We assume that $w\in C^1 (\mathbb{R}_+)$ is a non-decreasing function with $w(0)=1$ and such that $w^{-1} \in L^1 (\mathbb{R}_+)$.

We define the space $\Hi_w$ by
\begin{equation*}
\Hi_w = \{ f \in AC(\mathbb{R}_+,\U) \vert  \norm{f}_{w} < \infty \},
\end{equation*}
where $\norm{f}^2_{w} := \norm{f(0) }^2_{\U} + \int_0^{\infty} w(x) \norm{f' (x)}_{\U}^2 \,dx$. Further define the scalar product
\begin{equation*}
\langle f, g \rangle_{w}= \langle f (0), g(0) \rangle_{\U} + \int_0^{\infty} w(x)  \langle f' (x), g' (x) \rangle_{\U} \, dx
\end{equation*}
which obviously satisfies $\norm{f}^2_{w}=\langle f, f \rangle_{w}$.

It is already known that $(\Hi_w, \norm{\cdot }_{w})$ is a separable Hilbert space (\citep[Prop. 3.4.]{BE2015}). Additionally, we know from \citep[Lemma 3.8.]{BE2015} that the evaluation map $\delta_x$ is a bounded linear operator from $\Hi$ to $\U$. This allows us to show that the semigroup $(\mathcal{S}_t )_{t\geq 0}$ is strongly continuous and to identify its generator.
\begin{proposition}
The family $(\mathcal{S}_t )_{t\geq 0}$ is a $C_0$-semigroup in $\Hi_w$, the domain $\text{Dom}(\partial_x)$ of its generator $\partial_x$ is densely defined, satisfies
 $$ \text{Dom}(\partial_x) = \{ f\in \Hi_w | f'\in \Hi_w\} $$
 and its generator is given by
  $$ \partial_x f = f',\quad f\in \text{Dom}(\partial_x). $$
\end{proposition}
\begin{proof}
It was shown in \citep[Lemma 3.7.]{BE2015} that $( \mathcal{S}_t )_{t\geq 0}$ is strongly continuous. It then follows (see for example \citep[Thm 1.4.]{EngelNagel}) that the generator $\partial_x$ of $( \mathcal{S}_t )_{t\geq 0}$ is densely defined. Let $f\in \text{Dom}(\partial_x)$. Then $\partial_xf \in \Hi_w$ and
 $$ \partial_xf(r) = \lim_{t \searrow 0} \frac{\mathcal S_tf(r)-f(r)}{t} = \lim_{t\searrow 0}\frac{f(t+r)-f(t)}{r} $$
 which is the classical right-derivative. Since $f\in AC(\mathbb{R}_+,\U)$, Lebesgue's differentiation theorem yields that $f'$ is the derivative of $f$ $ds$-a.e., i.e.\ there is a set $N$ including $\{0\}$ of Lebesgue measure zero such that outside $N$ we find $f'(r) = \lim_{t\rightarrow 0} \frac{f(t+r)-f(r)}{t} = \partial_xf(r)$. Thus, $f' = \partial_xf$ $ds$-a.e. but $\partial_xf \in AC(\mathbb{R}_+,\U)$ and, hence, continuous. Consequently, $\partial_xf$ is a continuous version of $f'$, so $f'=\partial_xf$. This proves that
  $$ \text{Dom}(\partial_x) \subseteq  \{ f\in \Hi_w | f'\in \Hi_w\} $$
 and that $$ \partial_x f = f',\quad f\in \text{Dom}(\partial_x). $$
 
 Now let $f\in \Hi_w$ be such that $f'\in\Hi_w$. Then, $t\mapsto\mathcal S_tf'$, $t\geq 0$, is continuous and, hence $$ \Gamma(r):= f + \int_0^r \mathcal S_tf' dt $$ defines a $C^1$-function from $\mathbb R_+$ to $\U$ with $\Gamma(0) = f$ and $\Gamma'(t) = \mathcal S_tf'$. For $x\geq 0$ we see that
  $$ \delta_x(\Gamma(r)) = f(x) + \int_0^r f'(t+x) dt = f(r+x) = \delta_x(\mathcal S_rf) $$
 and, hence, we have $\Gamma(r) = \mathcal S_rf$. Consequently, $f\in\text{Dom}(\partial_x)$ and 
  $$ \partial_xf = \Gamma'(0) = f'. $$
  This concludes the proof.
\end{proof}

It remains to show that $( \mathcal{S}_t )_{t\geq 0}$ is quasi-contractive. The proof will make use of the adjoint operator $\delta_x^{*}$ of $\delta_x$, which we derive in Lemma~\ref{dx:adjoint}. For this we need the following result about the weak derivative of the scalar product.
\begin{lemma}\label{integral:scalarproduct}
Let $f\in \Hi_w$. Then, for every $u\in \U$
\begin{equation}\label{weak:der:scalar}
\langle f  (x),u \rangle_{\U} =\langle f(0), u \rangle_{\U} + \int_0^x  \langle f' (t), u \rangle_{\U} \, dt .
\end{equation}
\end{lemma}
\begin{proof}
Using that $w^{-1} \in L^1 (\mathbb{R}_+)$, it follows from \citep[Prop. 3.5.]{BE2015} that 
\begin{equation*}
f  (x)=f(0) + \int_0^x f' (t)\, dt
\end{equation*}
with $f' \in L^1 (\mathbb{R},\U)$ and the integral on the right hand side is in the sense of Bochner. This shows that 
\begin{equation*}
\langle f  (x),u \rangle_{\U} =\langle f(0), u \rangle_{\U} + \langle  \int_0^x   f' (t) \, dt, u \rangle_{\U} .
\end{equation*}
But since for every $u\in \U$ the operator $\langle \cdot ,u \rangle : \U \rightarrow \mathbb{R} $ is bounded and linear, we obtain that 
\begin{equation*}
\langle  \int_0^x   f' (t) \, dt, u \rangle_{\U} = \int_0^x  \langle f' (t), u \rangle_{\U} \, dt,
\end{equation*}
by properties of the Bochner integral. Thus, (\ref{weak:der:scalar}) follows.
\end{proof}
The last lemma allows us to derive the adjoint operator of the evaluation map $\delta_x$.
\begin{lemma}\label{dx:adjoint}
The adjoint operator $\delta_x^* : \U \rightarrow \Hi_w$ of $\delta_x$, $x\in\mathbb R_+$ is given by
\begin{equation*}
\delta_x^* (u) (\cdot) =(1 + \int_0^{\cdot\wedge x} w^{-1} (s) \, ds ) u
\end{equation*}
and $\norm{ \delta_x^* }^2_{\mathrm{op}}=\norm{ \delta_x}^2_{\mathrm{op}} = 1 + \int_0^{x} w^{-1} (s) \, ds  \leq 1 + x$.
\end{lemma}
\begin{proof}
Let $\delta_x^*$ be defined as above. First observe that by the integral representation of $\delta_x^*$, it follows that 
\begin{equation*}
\delta_x^{*} (u)' (t) = w^{-1} (t) 1_{\{t\leq x \}} \cdot u .
\end{equation*}
We need to show that $\langle f , \delta_x^* (u) \rangle_{w} =\langle \delta_x (f) , u \rangle_{\U}$. To see this calculate 
\begin{align*}
\langle f , \delta_x^* (u) \rangle_{w} &= \langle f (0)  , \delta_x^* (u)(0) \rangle_{\U}  + \int_0^{\infty} w(t)  \langle f' (t)  , \delta_x^{*} (u)'(t) \rangle_{\U} \, dt \nonumber \\
& = \langle f (0)  , u \rangle_{\U} + \int_0^{x}  \langle f' (t)  , u \rangle_{\U} \, dt \nonumber \\
& =  \langle f (x)  , u \rangle_{\U} \\&=  \langle \delta_x (f) , u \rangle_{\U}  \nonumber,
\end{align*}
where we used Lemma~\ref{integral:scalarproduct} in the second to the last line.
The norm calculates as
\begin{align*}
\norm{ \delta_x^* (u) }^2_{w} &= \langle u , u \rangle_{\U} + \int_0^x w(t) \langle  w^{-1}(t) u , w^{-1}(t) u \rangle_{\U} \, dt \nonumber \\
&= \norm{u}_{\U}^2 +  \norm{u}_{\U}^2 \int_0^x w^{-1}(t) \, dt\nonumber
\end{align*}
and since $w(t) \geq 1$ it follows that $\norm{ \delta_x^* (u) }_{w}^2  \leq (1+x) \norm{u}^2_{\U} $. Since $\norm{A}_{\mathrm{op}} = \norm {A^*}_{\mathrm{op}}$ for any adjoint operator $A^*$ of a linear operator $A$, it follows that $\norm{ \delta_x^* }_{\mathrm{op}}^2=\norm{ \delta_x}_{\mathrm{op}}^2$.
\end{proof}
With the help of the last lemma we can now show that the semigroup $( \mathcal{S}_t )_{t\geq 0}$ is quasi contractive.
 \begin{proposition}\label{p:opnorm S}
 The semigroup $( \mathcal{S}_t )_{t\geq 0}$ satisfies the operator-norm bound
 \begin{equation*}
 \norm{\mathcal S_t}_{\mathrm{op}} \leq e^{t / 2},\quad t\geq0.
 \end{equation*}
In particular, it is quasi-contractive. 
 \begin{proof}
The proof is a straightforward modification of a similar result of \citep{BK2017} in the case when $\mathcal U=\mathbb R$. 
We include the proof here for the convenience of the reader. 

Fix $t\geq 0$ and $f\in\mathcal H_w$. Define the functions $g(x)=f(t\wedge x)$ and
$\widetilde{g}(x)=\mathrm{1}_{t\leq x}(f(x)-f(t))$. Then it is easy to see that $g,\widetilde{g}\in\mathcal H_w$ are
orthogonal and $f=g+\widetilde{g}$. Moreover, $\|f\|_w^2=\|g\|_w^2+\|\widetilde{g}\|_w^2$. Since
$\mathcal S_tg(x)=g(x+t)=f(t\wedge(x+t))=f(t)=g(t)$, we find 
$$
\|\mathcal S_t g\|_w^2=\|g(t)\|_{\mathcal U}^2=\|\delta_tg\|_{\mathcal U}^2\leq\|\delta_t\|_{\mathrm{op}}^2\|g\|_w^2.
$$
But from Lemma~\ref{dx:adjoint} it holds that $\|\delta_t\|_{\mathrm{op}}^2\leq 1+t$, and 
hence, $\|\mathcal S_tg\|_w^2\leq (1+t)\|g\|_w^2$. On the other hand, it follows from the
non-decreasing property of $w$ and $\widetilde{g}(t)=0$ that, 
\begin{align*}
\|\mathcal S_t\widetilde{g}\|_w^2&=\|(\mathcal S_t\widetilde{g})(0)\|_{\mathcal U}^2+
\int_0^{\infty}w(x)\|(\mathcal S_t\widetilde{g})'(x)\|_{\mathcal U}^2\,dx \\
&=\|\widetilde{g}(t)\|_{\mathcal U}^2+\int_0^{\infty}w(x)\|\widetilde{g}'(x+t)\|_{\mathcal U}^2\,dx \\
&=\int_t^{\infty}w(y-t)\|\widetilde{g}'(y)\|_{\mathcal U}^2\,dy \\
&\leq\int_t^{\infty}w(y)\|\widetilde{g}'(y)\|_{\mathcal U}^2\,dy \\
&\leq \|\widetilde{g}\|_w^2\,.
\end{align*}
The constancy of $\mathcal S_tg$ and $\mathcal S_t\widetilde{g}(0)=\widetilde{g}(t)=0$ yield orthogonality of
$\mathcal S_tg$ and $\mathcal S_t\widetilde{g}$:
\begin{align*}
\langle\mathcal S_tg,\mathcal S_t\widetilde{g}\rangle_w&=\langle g(t),\widetilde{g}(t)\rangle_{\mathcal U}+\int_0^{\infty}w(x)\langle(\mathcal S_tg)'(x),(\mathcal S_t\widetilde{g})'(x)\rangle_{\mathcal U}\,dx=0\,.
\end{align*}
We therefore find,
$$
\|\mathcal S_t f\|_w^2=\|\mathcal S_tg+\mathcal S_t\widetilde{g}\|_w^2=\|\mathcal S_tg\|_w^2+\|\mathcal S_t\widetilde{g}\|_w^2\leq (1+t)\|g\|_w^2+\|\widetilde{g}\|_w^2.
$$
But as $t\geq 0$, $(1+t)\|g\|_w^2+\|\widetilde{g}\|_w^2\leq(1+t)(\|g\|_w^2+\|\widetilde{g}\|_w^2)=(1+t)\|f\|_w^2$,
and$(1+t)\leq\exp(t)$. Hence, $\|\mathcal S_t f\|_w^2\leq\exp(t)\|f\|_w^2$, and we conclude that 
$\|\mathcal S_t\|_{\mathrm{op}}\leq\exp(t/2)$. 
 \end{proof}
 \end{proposition}

\subsection{Conditions on the parameter functions $\mu$ and $\sigma$}\label{param:conditions}
Let us now look at sufficient conditions on the parameter functions $\mu$ and $\sigma$ in the
SVIE which ensure Lipschitz continuity and linear growth as required in Assumption~\ref{Ass:Lip}. We will assume that Assumption \ref{assump:2:functions} holds and write $\Hi_w(\mathbb R)$ when we replace $\U$ with $\mathbb R$ in the definition of $\Hi_w$, i.e.\ $\Hi_w(\mathbb R)$ is the space of absolutely continuous functions $f$ from $\mathbb R_+$ to $\mathbb R$ such that
 $$ \int_0^\infty (f'(x))^2 w(x) dx < \infty. $$

We will have to assume that 
    \begin{align*}
       x\mapsto\mu(x,t,u), \quad  x\mapsto\sigma(x,t,u)
    \end{align*}
is absolutely continuous and that they posses versions of their absolute continuous derivatives $\mu'$ and $\sigma'$ (w.r.t.\ their first variable) which are measurable as functions from $\mathbb R_+\times\mathbb R_+\times\U$ to $\U$ resp.\ $L(\U,\V)$.

\begin{proposition}
\label{prop:lipandgrowth-cond}
  Assume that there are $\ell_a,\ell_b\in \Hi_w(\mathbb R)$ with 
    \begin{align*}
\norm{\mu(t,t,0)}_{\mathcal U}&\leq \ell_a(0), \\
\norm{\mu'(t,s,0)}_{\mathcal U}&\leq \ell_a'(t-s), \\
\norm{\mu(t,t,u_1)-\mu(t,t,u_2)}_{\mathcal U}&\leq \ell_a(0)\norm{u_1-u_2}_{\mathcal U}, \\
      \norm{\mu'(t,s,u_1)-\mu'(t,s,u_2)}_{\mathcal U}&\leq \ell_a'(t-s)\norm{u_1-u_2}_\U, \\
\norm{\sigma(t,t,0)}_{\mathrm{op}}&\leq \ell_b(0), \\
\norm{\sigma'(t,s,0)}_{\mathrm{op}}&\leq \ell_b'(t-s), \\
\norm{\sigma(t,t,u_1)-\sigma(t,t,u_2)}_{\mathrm{op}}&\leq\ell_b\norm{u_1-u_2}_{\mathcal U},  \\
\norm{\sigma'(t,s,u_1)-\sigma'(t,s,u_2)}_{\mathrm{op}}&\leq\ell_b'(t-s)\norm{u_1-u_2}_{\mathcal U}
    \end{align*}
 for any $t,s\geq 0$, $u_1,u_2\in\U$. Then Assumption \ref{Ass:Lip} holds with $L_a=\|\ell_a\|^2_{\Hi_w(\mathbb R)}\|\delta_0\|_{\mathrm{op}}^2, L_b=\|\ell_b\|^2_{\Hi_w(\mathbb R)}\|\delta_0\|_{\mathrm{op}}^2, K_a=\|\ell_a\|^2_{\Hi_w(\mathbb R)}$ and $K_b=\|\ell_b\|^2_{\Hi_w(\mathbb R)}$.
\end{proposition}
\begin{proof}
  Let $t\geq 0$, $h_1,h_2\in\Hi_w$. Then we have
  \begin{align*}
     \| a(t,h_1) &- a(t,h_2) \|^2_{w} \\
       &= \|\mu(t,t,\delta_0h_1)-\mu(t,t,\delta_0h_2)\|_{\U}^2 + \int_0^\infty \left\|\mu'(t+s,t,\delta_0h_1)-\mu'(t+s,t,\delta_0h_2)\right\|_{\U}^2w(s) ds \\
      &\leq |\ell_a(0)|^2 \|\delta_0\|_{\mathrm{op}}^2\|h_1-h_2\|^2_{w} + \|\delta_0\|_{L(\Hi_w,\U)}^2\|h_1-h_2\|^2_{w}\int_0^\infty (\ell_a'(s))^2w(s) ds \\
      &= \|\ell_a\|^2_{\Hi_w(\mathbb R)}\|\delta_0\|_{\mathrm{op}}^2\|h_1-h_2\|^2_{w}.
  \end{align*}
 Also, we have
  \begin{align*}
     \| a(t,0) \|^2_{w} &= \|\mu(t,t,0)\|_{\U}^2 + \int_0^\infty \left\|\mu'(t+s,t,0)\right\|_{\U}^2w(s) ds \\
      &\leq \|\ell_a\|^2_{\Hi_w(\mathbb R)}.
  \end{align*}
  With similar arguments for $\sigma$ and $b$ we conclude that the Lipschitz and linear growth conditions are satisfied.
\end{proof}

In the next section we investigate homogeneous SVIEs and their invariant measures. By homogeneous we mean that 
$$ \mu(t,s,u) = \mu(t-s,0,u) =: \mu(t-s,u),\quad \sigma(t,s,u) = \sigma(t-s,0,u) =: \sigma(t-s,u)$$ for any $s,t\geq 0$, $u\in \U$.
We have the following corollary to Proposition \ref{prop:lipandgrowth-cond}:
\begin{corollary}
Assume that $\mu,\sigma$ are homogeneous and that there are $\ell_a,\ell_b\in\Hi_w(\mathbb R)$ with
     \begin{align*}
 \norm{\mu(0,u_1)-\mu(0,u_2)}_{\mathcal U}&\leq \ell_a(0)\norm{u_1-u_2}_\U, \\
 \norm{\mu'(t,u_1)-\mu'(t,u_2)}_{\mathcal U}&\leq \ell_a'(t)\norm{u_1-u_2}_\U, \\
 \norm{\sigma(0,u_1)-\sigma(0,u_2)}_{\mathrm{op}}&\leq \ell_b(0)\norm{u_1-u_2}_\U, \\
\norm{\sigma'(t,u_1)-\sigma'(t,u_2)}_{\mathrm{op}}&\leq\ell_b'(t)\norm{u_1-u_2}_{\mathcal U}
    \end{align*}
for any $t\geq 0$, $u_1,u_2\in\U$. Then Assumption \ref{Ass:Lip} holds.
\end{corollary}

\subsection{Limiting measure}
We are going to discuss two types of conditions corresponding to Proposition~\ref{prop:limit} and Proposition~\ref{prop:perm abstract} in the abstract setting that ensure a limiting measure for the SVIE. Both are written in terms of the long-term behavior of the coefficients, i.e.\ on $\lim_{t\rightarrow \infty} (\mu(t,\cdot),\sigma(t,\cdot))$. They are tailored to our specific choice of space $\Hi_w$ from Section \ref{examples:H} and make use of the fact that the elements in $\Hi_w$ have a 'value at infinity' which will allow us to identify these limits. This idea is adopted from  
Tehranchi \citep{Tehranchi2005}. In order to make this rigorous we need the following
\begin{lemma}\label{lem:infinity}
  Let $h\in \Hi_w$. Then $h'\in L^1(\mathbb R_+,\U)$.
\end{lemma}
\begin{proof}
Cauchy-Schwarz inequality yields
  \begin{align*}
\left(\int_{0}^{\infty} \norm{h'(x)}_{\U} dx \right)^2&=\left(\int_{0}^{\infty} \norm{h'(x)}_{\U} \frac{w(x)^{1/2}}{w(x)^{1/2}}dx \right)^2 \\
& \leq \left(\int_{0}^{\infty} \norm{h'(x)}_{\U}^2 w(x) dx \right)\left(\int_{0}^{\infty} w(x)^{-1} dx \right)<\infty.
\end{align*}
 Thus $h'\in L^1(\mathbb R_+,\U)$.
\end{proof}

\begin{proposition}\label{p:infinity}
  For any $h\in \Hi_w$ the limit of $h(t)$ for $t\rightarrow \infty$ exists, the linear map
   $$ \delta_\infty h := \lim_{t\rightarrow\infty} h(t) $$
  is an element of $L(\Hi_w,\U)$ and
   $$ \| h\|^2_{w,\infty} := \| h\|^2_{w} - \|h(0)\|_{\U}^2 + \|\delta_\infty(h)\|_{\U}^2 $$
  defines an equivalent Hilbert-space norm on $\Hi_w$ with scalar product
   $$ \<h,g\>_{w,\infty } = \<\delta_\infty h,\delta_\infty g\>_U + \int_0^\infty \<h'(s),g'(s)\>_\U w(s) ds,\quad f,g\in \Hi_w. $$
  Moreover, the operator norm of $\delta_0$ relative to the $\|\cdot\|_{w,\infty}$-norm is equal to $\sqrt{1+\int_0^\infty\frac1{w(s)}ds}$ and the operator norm of $\delta_0$ restricted to 
  $$ \Hi_w^0 := \{h\in\Hi_w: \delta_\infty h = 0\} $$
  is given by $\sqrt{\int_0^\infty\frac1{w(s)}ds}$.
\end{proposition}
\begin{proof}
   Lemma \ref{lem:infinity} implies that $h(t) = h(0) + \int_0^t h'(s) ds \rightarrow h(0) + \int_0^\infty h'(s) ds = \delta_\infty h$ for $t\rightarrow\infty$. $\delta_\infty$ is the everywhere defined pointwise limit of $\delta_t$ for $t\rightarrow\infty$ and, hence, the uniform boundedness principle yields that $\delta_\infty\in L(\Hi_w,\U)$ and in fact Lemma \ref{dx:adjoint} yields $\|\delta_\infty\|_{\mathrm{op}} = 1 + \int_0^\infty \frac{1}{w(s)}ds$. Obviously, $\| \cdot\|_{w,\infty}$ defines a Hilbert-space norm equivalent to $\|\cdot\|_{w}$ with
    $$ \| \cdot\|^2_{w,\infty} \leq \|\delta_\infty\|^2_{\mathrm{op}} \|\cdot\|^2_{w}, \quad  \| \cdot\|^2_{w} \leq \|\delta_\infty\|^2_{\mathrm{op}} \|\cdot\|^2_{w,\infty}$$
    where $\|\delta_\infty\|^2_{\mathrm{op}}$ denotes the operator norm of $\delta_\infty$ relative to the the $\|\cdot\|_w$-norm. 
    
  Now, we define 
   $$\phi:\U\rightarrow \Hi, u \mapsto \left(\mathbb R\ni x\mapsto (1+\int_x^\infty \frac{1}{w(s)}ds)u\right).$$ 
 We have
   $$ \<\phi(u),f\>_{w,\infty} = \<u,f(\infty)\>_\U - \int_0^\infty \<u,f'(s)\>_\U ds = \<u,f(0)\> = \<u,\delta_0f\>_\U $$
   for any $f\in \Hi_w$, $u\in\U$. Thus, $\phi = \delta_0^*$ relative to the $\<\cdot,\cdot\>_{w,\infty}$-scalar product. The operator-norm of $\delta_0$ in the $\|\cdot\|_{w,\infty}$-norm equals the operator norm of $\phi$ which is given by
    $$ \| \phi\|_{\mathrm{op}}^2  = 1 + \int_0^\infty \frac{1}{w(s)}ds.$$

For the last part we work on the smaller space $\Hi_w^0$ and define
 $$ \psi:\U\rightarrow \Hi_w^0, u \mapsto \left(\mathbb R\ni x\mapsto (\int_x^\infty \frac{1}{w(s)}ds)u\right).$$
  We have
    $$ \<\psi(u),f\>_{w,\infty} =\<\phi(u),f\>_{w,\infty} = \<u,\delta_0f\>_\U $$
  for any $f\in \Hi_w^0$ where the first equality follows from the fact that $f(\infty)=0$. The operator norm of $\psi$ equals the operator norm of $\delta_0$ restricted to $\Hi_w^0$ and, hence, we find the claimed formula for its operator norm.
\end{proof}
 We will use the notation $h(\infty):=\delta_\infty h$ for any $h\in \Hi_w$ and the semi-norm $\|h\|_0^2 := \|h\|_{w,\infty}^2 - \|h(\infty)\|_\U^2$ for $h\in \Hi_w$. Since $a(h)$ is assumed to be in $\Hi_w$ we may write $\mu(\infty,u) := a(\bar u)(\infty)$ where $\Hi_w\ni \bar u:\mathbb R_+\rightarrow\U, t\mapsto u$. Also we write $\sigma(\infty,u):=b(\bar u)(\infty)$ and both are simply the long-term limits, i.e. $\mu(\infty,u) = \lim_{t\rightarrow \infty}\mu(t,u)$ for $u\in \U$.

The advantage of the $\|\cdot\|_{w,\infty}$-norm lies in the fact that it orthogonalises the kernel of the generator of $\partial_x$ and the space $\Hi_w^0$ which is $\mathcal S$-invariant.
\begin{lemma}\label{lem:contraction}
  Let $\mathcal C\subseteq \Hi_w$ be the set of constant functions. Then $\mathcal C$ is orthogonal to $\Hi_w^0$ in the Hilbert space $(\Hi_w,\|\cdot\|_{w,\infty})$. Moreover, we have
   $$ \|\mathcal S_t|_{\Hi_w^0}\|_{\mathrm{op}} \leq e^{-\alpha_w/2 t} $$
   where $\alpha_w := \inf_{x\geq 0} \frac{w' (x)}{w(x)}\geq 0$ and we have
   $$ \|\mathcal S_t\|_{\mathrm{op}} \leq 1 $$
   where $\|\cdot\|_{\mathrm{op}}$ denotes the operator norm on $\Hi_w^0$ and $\Hi_w$ relative to the $\|\cdot\|_{w,\infty}$-norm respectively.
\end{lemma}
\begin{proof}
  Let $h\in \Hi_w^0$. We have
   \begin{align*}
      \|\mathcal S_th\|^2_{w,\infty} &= \int_0^\infty \|h'(s+t)\|_\U^2 w(s) ds \\
                                      &\leq e^{-\alpha_wt}\int_0^\infty \|h'(s+t)\|_\U^2 w(s+t) ds  \\
                                      & = e^{-\alpha_wt}\int_t^\infty \|h'(s)\|_\U^2 w(s) ds \\
                                      &\leq e^{-\alpha_wt} \|h\|_{0}
   \end{align*}
   for any $t\geq 0$. Moreover, for $f\in \Hi_w$ and $h := f - f(\infty)$ we find that
    \begin{align*}
       \|\mathcal S_tf\|^2_{w,\infty} &= \|f(\infty)\|_\U^2 +  \|\mathcal S_th\|^2_{w,\infty} \\
        &\leq \|f(\infty)\|_\U^2 + e^{-\alpha_wt} \|f\|_0^2 \\
        &\leq \|f\|_{w,\infty}^2
    \end{align*}
    for any $t\geq 0$.
\end{proof}

We can now state our main results for the existence of invariant laws for the SVIE \eqref{V:SDE:hom}. The first is about the case when the impact of a push in direction $a(h)$ or $b(h)$ vanishes over time. The second theorem, Theorem \ref{inv:distribution 2}, covers cases where these impacts do not vanish at infinity.
\begin{theorem}\label{inv:distribution}
Let $w\in C^1 (\mathbb{R}_+, \mathbb{R}_+)$, assume $\alpha_w =\inf_{x\geq 0} \frac{w' (x)}{w(x)}>0$ and that there are $L_a,L_b\geq 0$ with
\begin{align}
\norm{a(h_1) - a(h_2)}_{w,\infty} &\leq   L_a \norm{h_1-h_2}_{w,\infty}\label{Lip:1} \\
\norm{b(h_1) - b(h_2)}_{\mathrm{op}} &\leq  L_b \norm{h_1-h_2}_{w,\infty} \label{Lip:2} \\
L_b^2 + 2L_a &< \alpha_w
\end{align}
for any $h_1,h_2\in \Hi_w^0$. We also assume that 
  \begin{align*}
      \mu(\infty,u) = 0,\quad \sigma(\infty,u) = 0.
  \end{align*}
Then there exists a probability measure $\Gamma$ on $\U$, which depends on the law of $X(0)$, such that $P^{X(t)}$ converges weakly to $\Gamma$ as $t\rightarrow\infty$.
\end{theorem}
\begin{proof}
 We apply Proposition \ref{prop:limit} with $\mathcal C=\{h\in\Hi_w:\partial_xh=0\}$. The orthogonal complement of $\mathcal C$ in $(\Hi_w,\|\cdot\|_{w,\infty})$ is $\Hi_w^0$. The coefficient $a$ is $\Hi_w^0$-valued and $b$ is $L(\V,\Hi_w^0)$-valued due to the assumptions. Lemma \ref{lem:contraction} yields that $\|\mathcal S_t|_{\Hi_w^0}\|_{\mathrm{op}} \leq e^{-\alpha_w t/2}$ and conditions (iii) to (v) of Proposition \ref{prop:limit} are met with the given constants $L_a,L_b$.
\end{proof}
A common choice of weight function for the Filipovi\'c space is $w(x)=\exp(\alpha x)$ for some constant $\alpha>0$. Then in Theorem \ref{inv:distribution} we find $\alpha_w=\alpha>0$. This also demonstrates that the weight function puts restrictions on the Lipschitz constants, as
$L_b^2/2+L_a<\alpha$. The bigger we choose $\alpha$, the more generously we can choose Lipschitz functions,
but on the other hand the stronger assumptions we put on the asymptotic behaviour towards zero of the derivative of the elements in $\mathcal H_w$ as $x\rightarrow\infty$. Of course, for such choice of
$w$, we have that $w^{-1}(x)=\exp(-\alpha x)\in L^1(\mathbb R_+)$.

\begin{theorem}\label{inv:distribution 2}
  Let $w\in C^1(\mathbb R_+,\mathbb R_+)$, assume $\alpha_w =\inf_{x\geq 0} \frac{w' (x)}{w(x)}>0$ and that there are $L_a,L_b>0$ and that there is $\beta\in(0,\alpha_w/2]$ with
\begin{align*}
\norm{a(g) - a(h)}_{0} &\leq   L_a \norm{g-h}_{w,\infty} \\
\norm{b(g) - b(h)}_{\mathrm{op}} &\leq  L_b \norm{g-h}_{w,\infty} \\
\< \delta_{\infty}(a(g)-a(h)),g(\infty)-h(\infty)\>_\U &\leq -\beta \|g(\infty)-h(\infty)\|_\U^2
\end{align*}
for any $g,h\in\Hi_w$ with $2L_a+L_b^2 < 2\beta$. Then there is a probability measure $\Gamma$ on $\U$ which does not depend on the law of $X(0)$, such that $P^{X(t)}$ converges weakly to $\Gamma$ as $t\rightarrow\infty$.
\end{theorem}
\begin{proof}
  We like to apply Proposition \ref{prop:perm abstract} with the projectors
   \begin{align*}
     \pi_0&:\Hi_w\rightarrow \Hi_w, f\mapsto (x\mapsto f(\infty)), \\
     \pi_1&:\Hi_w\rightarrow \Hi_w, f\mapsto (x\mapsto f(x)-f(\infty)).
\end{align*}    
  Obviously, $\pi_0+\pi_1$ is the identity operator, they are orthogonal projections on $(\Hi_{w},\|\cdot\|_{w,\infty})$. As $\mathcal{S}_t$ acts as the identity on $\{h\in\Hi_w:\partial_xh=0\}$ and $\Hi^0_w$ is $\mathcal{S}_t$ invariant we get that
  $$ \|  \pi_0 \mathcal{S}_t  g \|_{\Hi_{w,\infty}} =  \|  \pi_0 \mathcal{S}_t (\pi_0 g) \|_{\Hi_{w,\infty}}=  \|  \pi_0 (g) \|_{\Hi_{w,\infty}} $$ 
 and (i) of Proposition~\ref{prop:perm abstract} holds with $\gamma = 0$ .
  Moreover 
  $$ \| \pi_1 \mathcal{S}_t g \|_{\mathcal{H}_{w,\infty}} =  \| \pi_1 \mathcal{S}_t (\pi_1 g )\|_{\mathcal{H}_{w,\infty}}  = \|  \mathcal{S}_t (\pi_1 g )\|_{\mathcal{H}_{w,\infty}} \leq e^{-\beta}  \|  \pi_1 g \|_{\mathcal{H}_{w,\infty}} $$
  for $0<\beta \leq \alpha_w/2 $ by Lemma~\ref{lem:contraction} and thus (ii) of Proposition~\ref{prop:perm abstract} holds. 
  
  Moreover, we have
 \begin{align*}
  &\<a(h)-a(g),\pi_0(h-g)\>_{w,\infty} = \<\delta_{\infty}(a(h)-a(g)),h(\infty)-g(\infty)\>_\U  \\
  &\leq  -\beta \|g(\infty)-h(\infty)\|_\U^2 =-\beta \|\pi_0(h-g)\|_{w,\infty}^2
\end{align*}
     by assumption and thus (iii) holds. Conditions (iv) and (v) follow directly form the Lipchitz conditions and the observation that $\delta_{\infty}\pi_1(h)=0$ and (vi) holds by assumption.
\end{proof}
We remark that while the conditions of Proposition~\ref{inv:distribution} and Proposition~\ref{inv:distribution 2} are stated in terms of the coefficient functions $a$ and $b$ using Proposition~\ref{prop:lipandgrowth-cond} it is straightforward to derive conditions purely based on the SVIE coefficients. The formulation here however is more general.

\subsection{Examples}
In the final part of this section we gather some examples. We start with a simple one, namely the classical mean-reverting Ornstein-Uhlenbeck process.

\begin{example}\label{ex:OU}
 We first consider the Ornstein-Uhlenbeck process on $\mathbb{R}$ defined by
\begin{equation}\label{OU:example}
X(t)=x_0+\int_0^t \lambda (\theta - X(s))\,ds+\int_0^t\sigma \,dW(s)\,,
\end{equation}
where $W$ is a Brownian motion. Here the coefficient functions $\mu$ and $\sigma$ are constant in the first argument and given by $\mu (t-s, u)= \lambda (\theta - u)$ and $\sigma (t-s,u)=\sigma\in\mathbb R$ for $0\leq s\leq t$, $u\in\mathbb R$ where $\lambda>0$, $\theta\in\mathbb R$ are constants. The corresponding functions on 
$\Hi_w(\mathbb R)$ are given by
 \begin{align*}
   a(h) &= (x\mapsto \lambda (\theta - \delta_0h)), \\
   b(h) &= (x\mapsto \sigma)
 \end{align*}
for $h\in \Hi_w(\mathbb R)$. However, due to the specification of $a$ and $b$ one can easily observe that the solution $Y$ to \ref{eq:hyp-spde} stays in $\{h\in\Hi_w:\partial_xh=0\}$ and one could also define $a$ as $a(h) : x\mapsto \lambda (\theta - \delta_{\infty}h)$. We see then that $a$ and $b$ satisfy the inequalities in Theorem \ref{inv:distribution 2} with $L_a = 0 = L_b$ and $0<\beta \leq \lambda$ where $\alpha_w:=\inf_{x\geq 0} \frac{w' (x)}{w(x)} >0$ is assumed to exist and to be positive (e.g. $w(x) = \exp(\rho x)$ for some $\rho>0$). 

If we had chosen the space $\Hi$ of constant functions in $\Hi_w(\mathbb R)$ with the trace norm instead, then we could apply Proposition \ref{prop:perm abstract} directly. We find that on this space $(\mathcal S_t)_{t\geq 0}$ is simply the identity and we may chose $\gamma=0$, $\pi_0$ equal to the identity, $\pi_1=0$, $\beta = \lambda$, $L_a = 0 = L_b$ and conditions (i) to (vi) are met. 
Condition (vii) of Proposition \ref{prop:perm abstract} reads as
 $$ 0 < 2\beta = 2\lambda, $$
i.e.\ again we find the invariant law for the OU process irrespective of the speed of mean-reversion.

\end{example}

\begin{example}
Our second example uses a generic separable Hilbert space $\U$, 
 \begin{align*}
     x_0 &\in \U, \\ 
     \mu(t,u) &:= \frac{1}{4}e^{-t}u, \\
     \sigma(t,u) &:= \frac{1}{4}e^{-t}u, \\
 \end{align*}
 for $t\geq 0$, $u\in\U$ and a $1$-dimensional mean zero and square-integrable L\'evy process $L$ with $\E|L(1)|^2 = 1$, i.e.\ the equation of interest is
  \begin{align*}
     X(t) &= x_0 + \int_0^t \mu(t-s,X(s)) ds + \int_0^t \sigma(t-s,X(s)) dL(s) \\
          &= x_0 + \int_0^t \frac{1}{4}e^{s-t}X(s) ds + \int_0^t \frac{1}{4}e^{s-t}X(s) dL(s).
\end{align*}   
 We define as usual
  \begin{align*}
     a(h) &:= \mu(\cdot,\delta_0h) = \frac{1}{4}e^{-(\cdot)}\delta_0h, \\
     b(h) &:= \sigma(\cdot,\delta_0h) = \frac{1}{4}e^{-(\cdot)}\delta_0h
  \end{align*}
  which for fixed function $h$ are in the space $\Hi_w$ with $w(x)=\exp(x)$. Clearly, $a$, $b$ satisfy the Lipschitz and linear growth condition of Theorem~\ref{Volterra:sol} and hence, the SVIE has a unique adapted \cadlag solution $X$. Also, $\sigma(\infty,u) = 0  =\mu(\infty,u)$ for any $u\in \U$ and, thus, we have what we may coin as temporary impact. We find $\alpha_w = 1$ and 
 \begin{align*}
    \|a(h)-a(g)\|^2_{w,\infty} &= \frac{1}{16}\|\delta_0(h-g)\|^2_\U \int_0^\infty e^{-2s}w(s)ds \\
          &= \frac{1}{16}\|h-g\|^2_{w,\infty} \|\delta_0|_{\Hi_w^0}\|^2_{\mathrm{op}}
\end{align*}   
for any $f,g\in\Hi_w^0$ where we used that according to Proposition \ref{p:infinity} the operator norm of $\delta_0$ on $\Hi_w^0$ equals 
 $$\sqrt{\int_0^\infty \frac1{w(s)}ds} = 1.$$ With the choice $L_a := L_b := \frac{1}{4}$ we find that
   $$ L_b^2 + 2L_a \leq \frac{9}{16} < 1 = \alpha_w $$
  and, hence, the requirements of Theorem \ref{inv:distribution} are met. Consequently, there is a limiting distribution for $X$.

\end{example}

\bibliographystyle{plainnat}
\bibliography{SVIE}

\end{document}